\documentclass{article}

\usepackage[T1]{fontenc} 
\usepackage[utf8]{inputenc}  
\usepackage[english]{babel}

\usepackage{amsmath}
\usepackage{amsthm}
\usepackage{amssymb}

\usepackage{authblk}

\usepackage{graphicx}
\usepackage{enumitem} 
\usepackage[dvipsnames]{xcolor}
\usepackage{todonotes}
\usepackage{soul}
\usepackage{hyperref}
\usepackage[capitalise]{cleveref}

\theoremstyle{plain}
\newtheorem{theorem}{Theorem}
\newtheorem{lemma}[theorem]{Lemma}
\newtheorem{corollary}[theorem]{Corollary}
\newtheorem{proposition}[theorem]{Proposition}

\theoremstyle{definition}
\newtheorem{definition}[theorem]{Definition}

\numberwithin{equation}{section}

\newcommand{\N}{\mathbb N}
\newcommand{\A}{\mathcal A}
\newcommand{\B}{\mathcal B}

\def\uu{\mathbf{u}}
\def\xx{\mathbf{x}}
\def\vv{\mathbf{v}}

\title{Reflection on the reflection complexity}

\author[1]{Lubom\'{i}ra Dvo\v{r}\'{a}kov\'{a}}

\author[1]{Edita Pelantov\'{a}}

\affil[1]{Department of Mathematics, Faculty of Nuclear Sciences and 
Physical Engineering, Czech Technical University in Prague\\
Trojanova 13, 120 00 Praha 2, Czech Republic\\
\url{lubomira.dvorakova,edita.pelantova@fjfi.cvut.cz}}

\begin{document}
\maketitle

\begin{abstract}

The factor complexity ${\mathcal C}_\uu$ of a sequence $\uu = u_0u_1u_2 \cdots$ over a finite alphabet counts the number of factors of length $n$  occurring in $\uu$, i.e., ${\mathcal C}_\uu(n) = \#\mathcal{L}_n(\uu)$, where $\mathcal{L}_n(\uu)= \{u_iu_{i+1}\cdots u_{i+n-1}: i \in \N\}$. 

Two factors of $\mathcal{L}_n(\uu)$ are  said to be equivalent if one factor is the reversal of the other one. 
Recently, Allouche et al. introduced the reflection complexity $r_\uu$ which counts the number of non-equivalent factors of $\mathcal{L}_n(\uu)$.  
They formulated the following conjecture: a sequence $\uu$ is eventually periodic if and only if $r_\uu(n+2) = r_\uu(n)$ for some $n \in \N$. Here we prove the conjecture and characterize the sequences for which  $r_\uu(n+2) = r_\uu(n)+1$ for every $n \in \N$ and also the sequences for which the equality is satisfied for every sufficiently large $n \in \N$.

\end{abstract}

\section{Introduction}   
The {\em factor complexity} $\mathcal{C}_\uu$ of a sequence $\uu = u_0u_1u_2\cdots$ over a finite alphabet $\mathcal{A}$ counts the number of distinct blocks of a given length occurring in the sequence~$\uu$. Vaguely speaking, the function $\mathcal{C}_\uu$ expresses the degree of disorder of letters in~$\uu$.  More precisely, $\mathcal{C}_\uu(n) = \#\mathcal{L}_n(\uu)$, where $\mathcal{L}_n(\uu)= \{u_iu_{i+1}\cdots u_{i+n-1}: i \in \N\}$.

The classical result of Morse and Hedlund~\cite{MorseHedlund1938} says that the factor complexity of a sequence $\uu$ which is not eventually periodic satisfies $\mathcal{C}_\uu(n+1)> \mathcal{C}_\uu(n)$ for every $n \in \N$. Morse and Hedlund studied a modification of the factor complexity, nowadays called the {\em abelian complexity}.  Two factors $u, v\in \mathcal{L}_n(\uu)$ are called abelian equivalent if the number of occurrences of a letter $a$ in $u$ and $v$ coincides for every letter $a\in \mathcal{A}$.  The abelian complexity counts the number of non-equivalent factors in  $\mathcal{L}_n(\uu)$; it is  discussed in details in the article \cite{RiSaZa2009}. When considering further equivalences on the set $\mathcal{L}_n(\uu)$, further modifications of the factor complexity can be introduced. Such an approach can be found in~\cite{ChPuZa2017} and \cite{Machacek2025}. A modification of the factor complexity of another kind is proposed by J. M. Cambell et al.~in~\cite{CCR2025}. The authors first apply a reduction to the sequence $\uu$ and then compute the factor complexity of the new sequence.

Here,  we study the {\em reflection complexity} $r_\uu$ as introduced in \cite{AlCaLiShSt2025}. Two factors of $\mathcal{L}_n(\uu)$ are  said to be equivalent if one factor is the reversal of the other one. 
For every $n \in \N$,  $r_\uu(n)$ is the number of non-equivalent factors in  $\mathcal{L}_n(\uu)$. In addition to a number of theoretical results, the article \cite{AlCaLiShSt2025} contains calculations of the reflection complexity for many known sequences. Based on these results and observations, several conjectures are formulated ibidem.  One of them (Conjecture 27) says: a~sequence $\uu$ is not eventually periodic if and only if 
 $r_\uu(n+2)> r_{\uu}(n) \ \text{for every $n \in \N$}$.

In this paper, we prove this conjecture and describe the sequences that are not eventually periodic and have the minimum reflection complexity.

The paper is organized as follows. In Section~\ref{sec:Preliminaries}, we recall  the necessary concepts and earlier results. Section ~\ref{sec:Reflection}   prepares auxiliary tools related to the reflection complexity. The proof of the conjecture itself is the content of Section~\ref{sec:Dukaz}. Aperiodic sequences having the minimum reflection complexity are characterized in Section~\ref{sec:minimal}. Section \ref{sec:comments} comments on another conjecture stated in~\cite{AlCaLiShSt2025} and opens some further questions.

\section{Preliminaries}\label{sec:Preliminaries}
An \textit{alphabet} $\mathcal A$ is a finite set of symbols, called \textit{letters}. 
A \textit{word} $u$ over $\mathcal A$ of \textit{length} $n$ is a finite string $u = u_0 u_1 \cdots u_{n-1}$, where $u_j\in\mathcal A$ for all $j \in \{0,1,\dots, n-1\}$. The length of $u$ is denoted $|u|$ and the set of all finite words over $\A$ is denoted $\A^*$. The set $\A^*$ equipped with concatenation as the operation forms a monoid with the \textit{empty word} $\varepsilon$ as the neutral element. Consider $u, p, s, v \in \A^*$ such that $u=pvs$, then the word $p$ is called a~\textit{prefix}, the word $s$ a~\textit{suffix} and the word $v$ a~\textit{factor} of $u$. 

A~\textit{(right-sided) sequence} $\uu$ over $\A$ is an infinite string $\uu = u_0 u_1 u_2 \cdots$ of letters $u_j \in \A$ for all $j \in \N$. Similarly, a~\textit{left-sided sequence} is an infinite string $\mathbf y=\cdots y_2 y_1 y_0$ of letters $y_j \in \A$ for all $j \in \N$. {A \textit{word} $w$ over $\mathcal A$ is called a~\textit{factor} of the sequence $\uu = u_0 u_1 u_2 \cdots$ if there exists $j \in \mathbb N$ such that $w = u_j u_{j+1} u_{j+2} \cdots u_{j+|w|-1}$. The integer $j$ is called an \textit{occurrence} of the factor $w$ in the sequence $\uu$. If $j=0$, then $w$ is a \textit{prefix} of $\uu$.

The \textit{language} $\mathcal{L}(\uu)$ of a sequence $\uu$ is the set of factors occurring in $\uu$. Obviously, $\mathcal{L}(\uu)= \bigcup_{n\in \N}{\mathcal L}_n(\uu)$. 
The language $\mathcal{L}(\uu)$ is called \textit{closed under reversal} if for each factor $w=w_0w_1\cdots w_{n-1}$, its \textit{reversal} (also called \textit{mirror image}) $\overline{w}=w_{n-1}\cdots w_1 w_0$ is also a factor of $\uu$.

A~factor $w$ of a sequence $\uu$ is \textit{left special} if $aw, bw \in \mathcal{L}(\uu)$ for at least two distinct letters ${a,b} \in \A$. A \textit{right special} factor is defined analogously. The set of left extensions is denoted $\mathrm{Lext}(w)$, i.e., $\mathrm{Lext}(w)=\{aw \in {\mathcal L}(\uu)\ : \ a\in {\mathcal A}\}$. Similarly, $\mathrm{Rext}(w)=\{wa \in {\mathcal L}(\uu) \ :\ a \in {\mathcal A}\}$. The set of both-sided extensions is denoted $\mathrm{Bext}(w)=\{awb \in {\mathcal L}(\uu)\ :\ a,b \in {\mathcal A}\}$.  Given a sequence $\uu$ and $n\in \mathbb N$, the first difference of factor complexity may be computed using the right extensions of factors in the following way, see~\cite{CassaigneRS1997},
$${\mathcal C}_\uu(n+1) - {\mathcal C}_\uu(n)  = \sum_{w \in \mathcal{L}_n(\uu)}\bigl(\#{\rm Rext}(w) - 1\bigr)\,.$$

A word $w$ is a \textit{palindrome} if $w$ is equal to its reversal, i.e., $w=\overline{w}$.
Similarly, the \textit{palindromic complexity} of a sequence $\uu$ is a mapping ${\mathcal P}_\uu:\mathbb N \to \mathbb N$, where $${\mathcal P}_\uu(n)=\#\{w \in {\mathcal L}_n(\uu) \ : \ w={\overline w}\}\,.$$ 

A sequence $\uu$ is \textit{recurrent} if each factor of $\uu$ has infinitely many occurrences in $\uu$. 
Moreover, a recurrent sequence $\uu$ is \textit{uniformly recurrent} if the distances between consecutive occurrences of each factor in $\uu$ are bounded. 
A~sequence $\uu$ is \textit{eventually periodic} if there exist words $w \in \A^*$ and $v \in \A^* \setminus \{\varepsilon\}$ such that $\uu$ can be written as $\uu = wvvv \cdots = wv^\omega$. If $\uu$ is not eventually periodic, $\uu$ is called \textit{aperiodic}.

A \textit{morphism} is a map $\psi: \A^* \to \B^*$ such that $\psi(uv) = \psi(u)\psi(v)$ for all words $u, v \in \A^*$.
The morphism $\psi$ can be naturally extended to a sequence $\uu=u_0 u_1 u_2\cdots$ over $\A$ by setting
$\psi(\uu) = \psi(u_0) \psi(u_1) \psi(u_2) \cdots\,$.



\section{The reflection complexity}\label{sec:Reflection}
Two words $u,v$ over an alphabet $\mathcal A$ are \textit{(reflectively) equivalent} if $u=v$ or $u=\overline{v}$, we write $u\sim_r v$.
The Czech words JELEN (deer in English) and NELEJ (don't pour in English) are equivalent. 
\begin{definition}
Let $\uu$ be a sequence. The \textit{reflection complexity} is a mapping $r_\uu: \mathbb N \to \mathbb  N$, where ${r_\uu}(n)$ is the number of distinct factors of $\uu$ of length $n$, up to equivalence by $\sim_r$.
\end{definition}
Denote $ \langle\uu\rangle_n $  the set of equivalence classes of words in $\mathcal{L}_n(\uu)$, i.e., $t \in \langle\uu\rangle_n $
means that $ t = \{w, \overline{w}\}\cap \mathcal{L}_n(\uu)$ for some $w \in \mathcal{L}(\uu)$.

\begin{definition}  
Let $\uu$ be a sequence and $n \in \mathbb{N}$. Let $t \in  \langle\uu\rangle_n $ and $w \in t$. By $\mathcal{T}(t)$ we denote the number of equivalence classes of the set 
${\rm Bext}(w) \cup {\rm Bext}(\overline{w})$.~\footnote{Recall that ${\rm Bext}(\overline{w})$ is empty if $\overline{w} \notin \mathcal{L}(\uu)$.}

\end{definition}

Obviously, 
\begin{equation}\label{eq:vypocetr(n+2)}
   r_\uu(n) = \sum_{t \in  \langle\uu\rangle_n } 1 \qquad \text{and}\qquad r_\uu(n+2) = \sum_{t \in  \langle\uu\rangle_n } \mathcal{T}(t). 
\end{equation}

Let us study the value ${\mathcal T}(t)$ for particular classes $t$.
\begin{lemma}\label{lem:simpleObservations}  Let $\uu$ be a sequence over an alphabet $\mathcal A$ and $n \in \mathbb{N}$. Let $t \in \langle\uu\rangle_n $ and $w\in t$. 
\begin{enumerate}
    \item If  $w$ is a prefix of $\uu$ and neither $w$ nor $\overline{w}$ have an occurrence  $j>0$, then $\mathcal{T}(t)=0$. Otherwise, $\mathcal{T}(t)\geq 1$.     
    \item  If  $cw$ is a right special factor of $\uu$ for some $c \in \mathcal{A}$. Then $\mathcal{T}(t)\geq \#{\rm Rext}(w) \geq 2$. Moreover,  
    $\mathcal{T}(t)= 2$ if and only if there exist two letters $a,b \in \mathcal{A}$, $a\neq b$, such that $\{cwa, cwb\}\subset {\rm Bext}(w)\cup{\rm Bext}(\overline{w})\subset \{cwa, cwb, a\overline{w}c, b\overline{w}c\}$. 

  \item  If  $wc$ is a left special factor of $\uu$ for some $c \in \mathcal{A}$. Then $\mathcal{T}(t)\geq \#{\rm Lext}(w) \geq 2$. Moreover,  
    $\mathcal{T}(t)= 2$ if and only if there exist two letters $a,b \in \mathcal{A}$, $a\neq b$, such that $\{awc, bwc\}\subset {\rm Bext}(w)\cup{\rm Bext}(\overline{w})\subset \{awc, bwc, c\overline{w}a, c\overline{w}b\}$.

\end{enumerate}

\end{lemma}
\begin{proof}
\begin{enumerate}
\item The first statement follows from the fact that for any factor $w$ of $\uu$ holds ${\rm Bext}(w)\cup{\rm Bext}(\overline{w})=\emptyset$ if and only if $w$ or $\overline w$ is a prefix of $\uu$ and neither $w$ nor $\overline{w}$ have an occurrence  $j>0$.
\item If $cw$ is a right special factor of $\uu$, then there exist $a,b \in \mathcal A$, $a\not =b$, such that $cwa, cwb \in {\mathcal L}(\uu)$. Clearly, $cwa$ and $cwb$ are not equivalent. It follows that $\mathcal{T}(t)\geq \#{\rm Rext}(w) \geq 2$. If $\mathcal{T}(t)=2$, then there are only two equivalence classes of the set ${\rm Bext}(w)\cup{\rm Bext}(\overline{w})$. The class containing $cwa$ is of the form $\{cwa\}$ or $\{cwa, a\overline{w}c\}$, similarly for the class containing $cwb$. Hence, $\{awc, bwc\}\subset {\rm Bext}(w)\cup{\rm Bext}(\overline{w})\subset \{awc, bwc, c\overline{w}a, c\overline{w}b\}$.
\item The proof is analogous to the proof of Item 2.
\end{enumerate}
\end{proof}

\begin{definition}\label{def:Index} Let $\uu=u_0u_1u_2 \cdots$ be a sequence. For each $n \in \mathbb{N}$, we denote ${\rm Ind}_\uu(n)$ the minimum $i \in \N$ having the property: 
  there exists $j >i$ such that the factors of length $n+1$ occurring in $\uu$ at positions $i$ and $j$ are equivalent. 
\end{definition}

For each $n\in \mathbb N$, ${\rm Ind}_\uu(n)$ is well defined. It suffices to realize that since $\uu$ is an infinite sequence and there are only finitely many factors of length $n+1$, there exists an index $i\in \mathbb N$ such that the factor of length $n+1$ occurring at $i$ has another occurrence $j >i$ in $\uu$ or its reversal has another occurrence $j >i$ in $\uu$. 
Obviously, if  ${\rm Ind}_\uu(n-1)>0$, then ${\rm Ind}_\uu(n)>0$, too.

\begin{lemma}\label{lem:prvniOpakovany} Let $\uu$ be a sequence and $n \in \mathbb{N}$. Let $v$ be the factor of length $n$ occurring at  $i={\rm Ind}_\uu(n)$.  
 If  $i>0$, then $\mathcal{T}(t_{\rm Ind})\geq 2$, where $t_{\rm Ind}$ is the equivalence class containing $v$.
    
\end{lemma}
\begin{proof} Let  $vd$ and $xv $ be the factors of length $n+1$ occurring at the positions $i$ and   $i-1$, respectively.  By the definition of $i$, the factors $xv$ and $\overline{v}x$ have no occurrence $\geq i$. 

Let $j>i$ be the index corresponding to $i={\rm Ind}_\uu(n)$ in Definition $\ref{def:Index}$. 
The index $j-1 >i-1$ is an occurrence of the factor $yvd$  or $j >i$ is an occurrence of $ d\overline{v}y$ for some $y\in \mathcal{A}$. Hence $y\neq x$.   Therefore, ${\rm Bext}(v)$ contains two non-equivalent factors $xvd, yvd $  or ${\rm Bext}(v)\cup {\rm Bext}(\overline{v})$ contains two non-equivalent factors $xvd, d\overline{v}y $.  Consequently,  $\mathcal{T}(t_{\rm Ind})\geq 2$. 
    
\end{proof}

The following corollary is proven in~\cite{AlCaLiShSt2025} as Theorem 23. We provide here another simple proof using our new formalism. 
\begin{corollary}  Let $\uu$ be a sequence. Then $r_{\uu}(n+2)\geq r_{\uu}(n)$ for every $n \in \N$.       
\end{corollary}
    
\begin{proof} 
The inequality holds trivially for $n=0$.

Let $i = {\rm Ind}_\uu(n-1)$, where $n\geq 1$. 

If $i=0$, then by Item 1 of Lemma \ref{lem:simpleObservations},   $\mathcal{T}(t)\geq 1$ for every class $t \in \langle \uu \rangle_n$. Hence \eqref{eq:vypocetr(n+2)} implies 
$r_{\uu}(n+2)\geq r_{\uu}(n)$ for $n\geq 1$. 

If $i={\rm Ind}_\uu(n-1)>0$, then ${\rm Ind}_\uu(n)>0$. By Lemma  \ref{lem:prvniOpakovany}, there exists an equivalence class $t_{\rm Ind}\in \langle \uu \rangle_n$ such that  $\mathcal{T}(t_{\rm Ind})\geq 2$. By Item 1 of Lemma \ref{lem:simpleObservations}, all classes $t\in \langle \uu \rangle_n$ except one have $\mathcal{T}(t)\geq 1$. Hence 
$r_{\uu}(n+2)\geq r_{\uu}(n)$ for $n\geq 1$. 
    
\end{proof}

\section{The reflection complexity and periodicity}\label{sec:Dukaz}

The aim of this section is to prove a conjecture on the characterization of eventually periodic sequences by reflection complexity; it is Conjecture 27 in~\cite{AlCaLiShSt2025}.
We prove it here as Theorem~\ref{thm:Conjectur27}. 
Let us recall a resembling characterization of eventually periodic sequences by factor complexity.

\begin{theorem}[\cite{MorseHedlund1938}]\label{thm:MorseHedlund} A~sequence $\uu$ is eventually periodic if and only if there exists $n \in \N$ such that ${\mathcal C}_\uu(n+1) ={\mathcal C}_\uu(n)$.  
\end{theorem}

Let us also recall a result from~\cite{AlCaLiShSt2025} on the relation of ${\mathcal C}_\uu$ and $r_\uu$ for eventually periodic sequences.
\begin{theorem}[\cite{AlCaLiShSt2025}, Theorem 39]\label{thm:ProPeriodic}  A sequence $\uu$ is eventually periodic if and only if both sequences $(r_\uu(2n))_{n\in \N}$ and $(r_\uu(2n + 1))_{n\in \N}$ are eventually constant.    
\end{theorem}

The following proposition seems to be ``folklore'' in combinatorics on words. Nevertheless, we have not found its proof anywhere. For reader's convenience, we provide its proof in Appendix.

 \begin{proposition}\label{pro:RSaLS} Let $\uu$ be an aperiodic sequence over an alphabet $\mathcal{A}$. 
 
 \begin{enumerate}
 
 \item 
 There exists a right-sided sequence ${\bf x}_L = x_0x_1x_2 \cdots$ and two different letters $a_L,b_L \in \mathcal{A}$ such that for every prefix $f$  of   ${\bf x}_L$    both factors $a_Lf$ and $b_Lf$ occur infinitely many times in $\uu$.  
 
 \item 
 There exists a left-sided sequence ${\bf y}_R = \ldots y_2y_1y_0 $ and two different letters $a_R,b_R \in \mathcal{A}$ such that for every suffix $g$  of   ${\bf y}_R$    both factors $ga_R$ and $gb_R$ occur infinitely many times in $\uu$.  
\end{enumerate}
 
 \end{proposition}

\begin{theorem}\label{thm:Conjectur27}\label{thm:Conjecture25}  A~sequence $\uu$ is eventually periodic if and only if there exists $n \in \N$ such that $r_{\uu}(n+2) = r_{\uu}(n)$. 
\end{theorem}
\begin{proof} 
The implication~($\Rightarrow$) was proven in~\cite{AlCaLiShSt2025}, see Theorem \ref{thm:ProPeriodic}. 

Let us prove the opposite implication ($\Leftarrow$). 

If $r_{\uu}(2) = r_{\uu}(0)=1$, then the unique class $t$ in $\langle\uu\rangle_2$   is either $t=\{aa\}$ or $t=\{ab,ba\}$ for some $a,b \in \mathcal{A}$. Thus  $\uu=a^{\omega}$ or $\uu=(ab)^{\omega}$, respectively. 

\medskip

\noindent If $r_{\uu}(n+2) = r_{\uu}(n)$ for some $n \in \N, n\geq 1$, let us proceed by contradiction, i.e., we assume that $\uu$ is aperiodic.      

Denote by $wc$ the prefix of ${\bf x}_L$ of length $n+1$, where  the  sequence ${\bf x}_L$ is as in   Proposition \ref{pro:RSaLS}  and $c \in \mathcal{A}$.  
Item 3 of Lemma  \ref{lem:simpleObservations} says that $\mathcal{T}(t_L)\geq 2$, where $t_L$ is the class containing $w$. 

\medskip

Let $i = {\rm Ind}_\uu(n)$.
\begin{description}
\item[Case $i =0$.]   \quad Clearly, ${\rm Ind}_\uu(n-1)=0$. Then, Item 1 of Lemma  \ref{lem:simpleObservations} gives  $\mathcal{T}(t)\geq 1$ for every $t \in \langle\uu\rangle_n$. As $\mathcal{T}(t_L)\geq 2$, 
Equation  \eqref{eq:vypocetr(n+2)} implies $r_{\uu}(n+2) \geq  r_{\uu}(n)+1$ --  a contradiction. 

\item[Case $i >0$.]  \quad  Let $v$ be the factor of length $n$ occurring at the position $i$.  
 By Lemma \ref{lem:prvniOpakovany},  $\mathcal{T}(t_{\rm Ind})\geq 2$ for the equivalence class $t_{\rm Ind}$ containing~$v$.
 
 If $t_{\rm Ind}\neq t_L$,  then $r_{\uu}(n+2) \geq  r_{\uu}(n)+1$,  since  all equivalence classes $t$ except one have $\mathcal{T}(t)\geq 1$  and at least two classes  have  $\mathcal{T}(t)\geq 2$ --  a~contradiction.  

 If $t_{\rm Ind}=t_L$,  then $v=w$ or $v= \overline{w}$.  Denoting $x = u_{i-1}$ and $d$ as in the proof of Lemma~\ref{lem:prvniOpakovany},  we see that $\{xvd, a_Lwc, b_Lwc\}\subset {\rm Bext}(w) \cup {\rm Bext}(\overline{w})$.

 By the definition of $i$,  the factors   $xv$  and $\overline{v}x$ have no occurrence $\geq i$,   whereas $a_Lwc$ and  $b_Lwc$ have infinitely many occurrences in $\uu$. Hence, $xvd$ cannot be equivalent to $a_Lwc$ or  to $b_Lwc$.   Moreover, $a_Lwc$ is not equivalent to  $b_Lwc$  as  $a_L \neq b_L$.    Altogether,  the factors $xvd, a_Lwc, b_Lwc$ belong to three different equivalence classes, i.e., $\mathcal{T}(t_L)\geq 3$. It yields a~contradiction,  too.

\end{description}

\end{proof}

\section{Aperiodic sequences with the minimum reflection complexity}\label{sec:minimal}

 Aperiodic sequences with the minimum factor complexity are called Sturmian sequences. We will show in the sequel that they also play an important role in the description of aperiodic sequences $\uu$ with the minimum reflection complexity, i.e., where $r_\uu(n+2)=r_\uu(n)+1$.

\begin{definition} A sequence $\uu$ is 
called 
\begin{itemize} \item Sturmian if ${\mathcal C}_\uu(n)= n+1$ for every $n \in \N$.

\item  quasi-Sturmian if there exist  $n_0 \in \N$   and a constant $c \in \N $ such that  ${\mathcal C}_\uu(n)= n+c$ for every $n \in \N, n\geq n_0$.    

\end{itemize}
\end{definition}

It is known that every Sturmian sequence $\uu$ has the language closed under reversal. 
Let us recall some equivalent characterizations of Sturmian sequences. 

\begin{theorem}[\cite{DrPi1999}]\label{thm:SturmianEquivalence} Let $\uu$ be a sequence. The following statements are equivalent. 
\begin{enumerate}
\item $\uu$ is Sturmian;
\item $\uu$ is binary and contains exactly one right special factor of every length;
\item ${\mathcal P}_\uu(n)+{\mathcal P}_\uu(n+1)=3$ for every $n\in \mathbb N$.
\end{enumerate}
\end{theorem}

 Let us also recall an equivalent characterization of quasi-Sturmian sequences.

\begin{theorem}[\cite{Cassaigne1997}]\label{thm:cassaigne} A sequence $\uu$ over an alphabet $\mathcal{A}$ is quasi-Sturmian if and only if $\uu$ can be written as $\uu = p \varphi({\bf v})$, where $p \in \mathcal{A}^*$,  $\vv$ is a Sturmian sequence over $\{a,b\}$ and  $\varphi:  \{a,b\}^* \mapsto \mathcal{A}^*$  is a non-periodic morphism (non-periodic means that $\varphi(ab)\neq \varphi(ba)$). 
\end{theorem}

\begin{theorem}[\cite{AlCaLiShSt2025},  Theorem 41]\label{thm:Theorem35Unich} Let $\uu$ be an aperiodic sequence.
\begin{enumerate}
    \item  For all $n \in \N$, we have $r_\uu(n)\geq 1+\lfloor  \frac{n+1}{2}\rfloor$.
    \item The sequence  $\uu$ is Sturmian if and only if  $r_\uu(n)= 1+\lfloor  \frac{n+1}{2}\rfloor$  for all $n \in \N$.
\end{enumerate}
\end{theorem}
Clearly,  $r_\uu(n)= 1+\lfloor  \frac{n+1}{2}\rfloor$  for all $n \in \N$ implies that $r_\uu(n+2)=r_\uu(n) +1$ for all $n \in \N$. 
But the reverse implication is not true as illustrated in the following lemma. 

\begin{lemma}\label{lem:ternary}
Let $\vv$ be a Sturmian sequence over $\{a,b\}$. Consider $\uu = \pi(\vv)$, where $\pi$ is a morphism given by 
$\pi: a \mapsto ac, \ b\mapsto bc$ with $c \notin \{a,b\}$. Then $r_\uu(n+2) = r_\uu(n)+1$ for every $n \in \N$.
\end{lemma}

\begin{proof}
Since $\vv$ is a Sturmian sequence and therefore has the language closed under reversal, ${\mathcal L}(\uu)$ is also closed under reversal and $r_\uu(n)=\frac{1}{2}\left({\mathcal C}_\uu(n)+{\mathcal P}_\uu(n)\right)$ by Item (b) of Theorem 9 in~\cite{AlCaLiShSt2025}. It thus suffices to determine the factor and palindromic complexity of $\uu$.
\begin{itemize}
 \item Consider $n=2k$, where $k\in \mathbb N, k\geq 1$. Each factor $u$ of length $n$ in $\uu$ is obtained from a factor $v$ of $\vv$ of length $k$, where $v=v_iv_{i+1}\cdots v_{i+k-1}$, as $u=cv_icv_{i+1}\cdots cv_{i+k-1}$ or $u=v_icv_{i+1}\cdots cv_{i+k-1}c$. Consequently, ${\mathcal C}_\uu(n)=2(k+1)$. Obviously, there is no palindrome of length $n$ in $\uu$.
 Altogether, $r_\uu(n)=\frac{1}{2}\left(2(k+1)+0\right)=k+1=\frac{n}{2}+1$. It holds for $n=0$, too.
 \item Consider $n=2k+1$, where $k\in \mathbb N, k\geq 1$. Each factor $u$ of length $n$ in $\uu$ is obtained from a factor $v$ of $\vv$ of length $k$, where $v=v_iv_{i+1}\cdots v_{i+k-1}$, as $u=cv_icv_{i+1}\cdots cv_{i+k-1}c$, or from a factor $v$ of $\vv$ of length $k+1$, where $v=v_iv_{i+1}\cdots v_{i+k}$, as $u=v_icv_{i+1}\cdots cv_{i+k}.$ Moreover, $u$ is a palindrome if and only if $v$ is a palindrome.  Consequently, ${\mathcal C}_\uu(n)=(k+1)+(k+2)=2k+3$.  
Using Item~3 of Theorem~\ref{thm:SturmianEquivalence}, ${\mathcal P}_\uu(n)={\mathcal P}_\vv(k)+{\mathcal P}_\vv(k+1)=3$. 
Therefore, $r_\uu(n)=\frac{1}{2}(2k+3+3)=k+3=\frac{n+1}{2}+2$. It holds for $n=1$, too.
\end{itemize}

 Putting the above results together, we obtain $r_\uu(n+2) = r_\uu(n)+1$ for every $n \in \N$.
\end{proof}

The following lemmata aim for a complete characterization of the sequences $\uu$ satisfying $r_\uu(n+2) = r_\uu(n)+1$ for every $n \in \N$. Such a characterization is provided in~Theorem~\ref{thm:r+1}.

\begin{lemma}\label{lem:LevyPravySpecial} 
Let $wc$ be a left special factor in $\uu$ and $eu$ be a right special factor in $\uu$  such that $u$ and $w$  belong to the same equivalence class $t$. 

If  $\mathcal{T}(t) = 2$, 
then $ u = \overline{w} $, $c=e$ and there exist letters $a,b \in \mathcal{A}$, $a\neq b$, such that 
\begin{equation}\label{eq:celyBext}{\rm Bext}({w})\cup {\rm Bext}(\overline{w}) = \{awc, bwc, c \overline{w}a, c \overline{w}b \} .\end{equation}
\end{lemma}
\begin{proof} 
We apply Lemma \ref{lem:simpleObservations} to deduce that $wc$ has exactly two left extensions, say $a_L$ and $b_L$,  and similarly, $eu$ has two right extensions $a_R$ and $b_R$. 

    Let us show that  $u$ equals $\overline{w}$. Indeed, as otherwise, $u=w \neq \overline{w}$  and by Item 3 of Lemma \ref{lem:simpleObservations},  $w$ would have just one right extension, namely $c$. It is impossible as $u$ is a right special factor. Combining Items 2 and 3 of  Lemma \ref{lem:simpleObservations}, we get 
\begin{enumerate}    
\item $ {\rm Bext}(w)\cup{\rm Bext}(\overline{w})\subset \{a_Lwc, b_Lwc, c\overline{w}a_L, c\overline{w}b_L\}$, 

\item ${\rm Bext}(w)\cup{\rm Bext}(\overline{w})\subset \{e\overline{w}a_R, e\overline{w}b_R, a_Rwe, b_Rwe\}$,

\item $\{a_Lwc, b_Lwc\}\cup \{e\overline{w}a_R, e\overline{w}b_R\}\subset {\rm Bext}(w)\cup{\rm Bext}(\overline{w})$. 
 \end{enumerate}   

If $c$ were not equal to $e$, then the inclusions 1. and 3. above would force $\{e\overline{w}a_R, e\overline{w}b_R\} \subset \{a_Lwc, b_Lwc\}$, which is impossible,  as
 $a_R\neq b_R$.  Hence $c=e$ and $ \{a_L,b_L\} = \{a_R,b_R\}$. 
    \end{proof}


 




\begin{lemma}\label{lem:RSdelkyN+1} Let $\uu$ be an aperiodic sequence over an alphabet $\mathcal{A}$  and $n \in \N$.     
If  $r_\uu(n+2) =  r_\uu(n)+1$, then  
$\uu$ contains only one  right special factor of length $n+1$ and  this factor  has exactly two right extensions. Moreover, the reversal of the right special factor is a left special factor of $\uu$ and both of those special factors occur infintely many times in $\uu$. 
\end{lemma}

\begin{proof} 
If $n=0$, i.e., $r(2)=r(0)+1=2$, it is not difficult to see that either there exist two letters $a,b$ such that $aa,ab,ba$ are all factors of length $2$ in $\uu$ or there exist three letters $a,b,c$ such that $ac, bc, ca, cb$ are all factors of length $2$ in $\uu$. The statement clearly holds in this case.

In the sequel, consider $n\geq 1$. 
Denote $eu$ the suffix of length $n+1$ of ${\bf y}_R$ and  $wc$ the prefix  of length $n+1$ of ${\bf x}_L$. By Proposition~\ref{pro:RSaLS}, both $eu$ and $wc$ occur infinitely many times in $\uu$.  By $t_R$ and $t_L$ we  denote the equivalence classes containing $u$ and  $w$, respectively.    By Items 2 and 3 of Lemma \ref{lem:simpleObservations},  $\mathcal{T}(t_L)\geq 2$ and $\mathcal{T}(t_R)\geq 2$. 

Assume that $hf$ with $h \in \mathcal{A}$ is a right special factor of $\uu$ of length $n+1$,  i.e., for some letters $A,B \in \mathcal{A}, A\neq B$, the factors  $hfA$ and $hfB$ belong to the language of $\uu$. If $t_0 \in \langle \uu \rangle_n$   stands for the class containing $f$, we can write  
\begin{equation}\label{eq:libovolnyRS} \{hfA, hfB\}\subset {\rm Bext}(f) \quad \text{ and }\ \mathcal{T}(t_0)\geq 2. 
\end{equation}

\medskip

Our aim is to show that  $hf = c \overline{w} = eu $. 
Denote $i = {\rm Ind}_\uu(n)$.   

\begin{description}
    \item[Case $i=0$.]  \ \ Clearly, ${\rm Ind}_\uu(n-1)=0$. Then, Item 1 of Lemma  \ref{lem:simpleObservations} gives  $\mathcal{T}(t)\geq 1$ for every $t \in \langle\uu\rangle_n$. Equality \eqref{eq:vypocetr(n+2)}  forces   $\mathcal{T}(t_L) =\mathcal{T}(t_R) =2$ and $t_L=t_R$.  Moreover, $\mathcal{T}(t)=1$ for every class $t \neq t_L,t_R $. Lemma \ref{lem:LevyPravySpecial}  implies that $c=e$ and $u=\overline{w}$. Therefore $c\overline{w}$ is the unique right special factor of length $n+1$ and it has exactly two right extensions by \eqref{eq:celyBext}. Hence, $hf =  c \overline{w}=eu$.

\item[Case $i>0$.] \ \  Let $v$ be  the factor occurring at the position $i = {\rm Ind}_\uu(n)$, let $t_{\rm Ind}$ denote  the equivalence class containing $v$ and $x=u_{i-1}$.  Recall that by Lemma~\ref{lem:prvniOpakovany}, $\mathcal{T}(t_{\rm Ind})\geq 2$.

If \  $t_{\rm Ind}\neq t_L$ and $t_{\rm Ind}\neq t_R$, then $t_R=t_L$, $\mathcal{T}(t_R)=\mathcal{T}(t_L) =\mathcal{T}(t_{\rm Ind}) = 2$ and 
 $\mathcal{T}(t)\leq 1$ for other classes. 

Let us explain that $t_0\neq t_{\rm Ind}$. 

As $\mathcal{T}(t_{\rm Ind}) = 2$, the definition of ${\rm Ind}_\uu(n)$ gives  
 ${\rm Bext}(v)\cup{\rm Bext}(\overline{v})\subset \{xvd, yvd, d\overline{v}y\}$ for some letters $y,d$ such that $y\neq x$. On the other hand,   $\{hfA, hfB\}\subset {\rm Bext}(f)$.  Obviously,  $ {\rm Bext}(f)$ cannot be a subset of ${\rm Bext}(v)\cup{\rm Bext}(\overline{v})$,  hence $t_0\neq t_{\rm Ind}$.  Consequently, $t_0 = t_R$ and  the reasoning  is the same as in the previous case.

If \  $t_{\rm Ind}=t_L$, then similarly as in the proof of Theorem~\ref{thm:Conjectur27}, we get $\mathcal{T}(t_L)\geq 3$. Hence  $t_0=t_L=t_R$,  \ $\mathcal{T}(t_L)= 3$ and   $\mathcal{T}(t)\leq 1$ for all other classes $t \in \langle \uu\rangle_n$. The fact that $\mathcal{T}(t_L)= 3$ forces \begin{equation}\label{eq:hodne}{\rm Bext}(w)\cup {\rm Bext}(\overline{w})\subset \{xvd, awc, bwc, c \overline{w}a, c \overline{w}b \}.\end{equation}

The factor  $xv$ has only one occurrence in $\uu$, whereas  $ hf$ occurs at least twice (as $hfA$ and $hfB$). Hence, $xv \neq  $  $hf$ and $hf = c \overline{w}$.

\end{description}

It is readily seen that in all cases, the reversal $wc$ of the right special factor $c\overline{w}$ is a left special factor of $\uu$.
\end{proof}

\begin{theorem}\label{thm:r+1} Let $\uu$ be a sequence. Then $r_\uu(n+2) = r_\uu(n)+1$ for every $n \in \N$ if and only if
\begin{itemize}
    \item either $\uu$ is a Sturmian sequence,

\item or there exist a Sturmian sequence $\vv$ over $\{a,b\}$ and a letter $c \notin \{a,b\}$  such that 
 $\uu = \pi(\vv)$ or $\uu = c\pi(\vv)$, where $\pi$ is a morphism given by 

\medskip 
\centerline{$\pi: a \mapsto ac, \ b\mapsto bc$. }
  
\end{itemize}

\end{theorem}

\begin{proof} 
$(\Rightarrow)$ \ 
If $\uu$ is binary and $r_\uu(n+2) = r_\uu(n)+1$ for every $n \in \N$, then $\uu$ is Sturmian by Lemma~\ref{lem:RSdelkyN+1} and Item 2 of Theorem~\ref{thm:SturmianEquivalence}.
Let us consider the case of $\uu$ defined over $\mathcal A$ with $\#\mathcal A\geq 3$.
By Theorem \ref{thm:Conjectur27}, $\uu$ is aperiodic. 
Since $r(2)=r(0)+1=2$, it is not difficult to see that there exist three letters $a,b,c$ such that $\{ac, bc, ca, cb\}$ is the set of factors of length $2$ in $\uu$.
Consequently, $\uu=cv_0cv_1cv_2\cdots$ or $\uu=v_0cv_1cv_2c\cdots$, where $\vv=v_0v_1v_2\cdots$ is an aperiodic binary sequence. 
Assume $\vv$ is not Sturmian.
It means that there exists $n$ such that ${\mathcal C}_\vv(n+1) - {\mathcal C}_\vv(n)\geq 2$. As $\vv$ is binary, there exist at least two distinct right special factors of length $n$ in $\vv$. Find the minimum $n$ with this property and denote $u,w$ the distinct right special factors of length $n$. 
 Minimality of $n$  and the fact that a suffix of a right special factor is a right special factor as well implies, without loss of generality,  
 $u=av'$ and $w=bv'$, where $v'$ may be empty. Clearly, $av'a, av'b, bv'a, bv'b$ are factors of $\vv$ and $ac\pi(v')a, ac\pi(v')b, bc\pi(v')a, bc\pi(v')b$ are factors of $\uu$. Consequently, there are two distinct right special factors $ac\pi(v')$ and $bc\pi(v')$ in $\uu$, which is in contradiction to Lemma~\ref{lem:RSdelkyN+1}. 

 $(\Leftarrow)$ \ If $\uu$ is a Sturmian sequence, then the equality $r_\uu(n+2) = r_\uu(n)+1$ for every $n \in \N$ follows from Item 2 of Theorem~\ref{thm:Theorem35Unich}. 
 If  $\uu = \pi(\vv)$ or $\uu = c\pi(\vv)$, where $\pi$ is a morphism given by 
$\pi: a \mapsto ac, \ b\mapsto bc$ and $\vv$ is a Sturmian sequence over $\{a,b\}$, then the equality $r_\uu(n+2) = r_\uu(n)+1$ for every $n \in \N$ follows from Lemma~\ref{lem:ternary}.
  
\end{proof}

\bigskip
In the sequel, our goal is to characterize the sequences $\uu$ satisfying $r_\uu(n+2) = r_\uu(n)+1$ for every sufficiently large $n \in \N$.
\begin{theorem}[\cite{AlCaLiShSt2025}, Theorem 43]\label{thm:OniQuasiSturm} Let $\uu$  be a quasi-Sturmian sequence. 
\begin{enumerate}
    \item  If a suffix of $\uu$ has the language closed under reversal, then $r_{\uu}(n) = \frac{n}{2} + \mathcal{O}(1)$; 
    \item  If no suffix of $\uu$ has the language closed under reversal, then $r_{\uu}(n) = n + \mathcal{O}(1)$.
\end{enumerate}

\end{theorem}

\begin{theorem}\label{thm:quasiSturmian} Let $\uu$ be a sequence. 
The equality $r_\uu(n+2) = r_\uu(n)+1$  takes place for all sufficiently large $n$ if and only if  $\uu$ is quasi-Sturmian and the language of  a suffix of $\uu$ is closed under reversal.
\end{theorem}
\begin{proof} ($\Rightarrow$) \ By Lemma \ref{lem:RSdelkyN+1}, there exists $n_0\in \mathbb N$ such that for every $n \geq n_0$, there is only one right special factor in ${\mathcal L}_n(\uu)$ and it has two right extensions. It implies that ${\mathcal C}_{\uu}(n+1) - {\mathcal C}_\uu(n) = 1$ for $n \geq n_0$. Hence $\uu$ is quasi-Sturmian. 
By Theorem \ref{thm:cassaigne}, a suffix of $\uu$, say $\uu'$,  is a morphic image of a Sturmian sequence. Consequently, $\uu'$ is a uniformly recurrent sequence. Thus, for every factor $u$ of $\uu'$ there exists $n\geq n_0$ such that $u$ occurs in the unique right special factor of length $n$. By Lemma~\ref{lem:RSdelkyN+1}, the reversal of the right special factor belongs to the language of $\uu$ and occurs infinitely many times in $\uu$, hence it is a factor of $\uu'$. Therefore $\overline{u}$ is a factor of $\uu'$, too.


\medskip

($\Leftarrow$) \  By Item 1 of Theorem \ref{thm:OniQuasiSturm}, there exists a finite set $C \subset \frac12 \mathbb{Z}$ such that $r_\uu(n) = \frac{n}{2} +c_n$ with $c_n \in C$ for each $n \in \N$.  The aperiodicity of $\uu$ and Theorem \ref{thm:Conjecture25} imply 
\begin{equation}\label{eq:quasi}\frac{n+2}{2}+ c_{n+2} = r_\uu(n+2)\geq r_\uu(n)+1 =\frac{n}{2}+ 1 + c_{n}. \end{equation}
It forces $c_{n+2}\geq c_n$ for every $n \in \N$. As the set $C$
 is finite, the sequences $(c_{2k})_{k \in \N}$  and $(c_{2k+1})_{k \in \N}$ are eventually constant. In other words, the inequality in \eqref{eq:quasi}  changes to equality starting from some index $n_1$.  

 \end{proof}

\section{Comments and questions}\label{sec:comments}

In addition to Conjecture 27 we have proved here, the article  \cite{AlCaLiShSt2025} also states several other conjectures and open questions. 
Let us comment on one more conjecture of $\cite{AlCaLiShSt2025}$.

\medskip
\noindent {\bf Conjecture 28:}  \  Let $\xx$ be a sequence of at most linear factor complexity. Then  the sequence $\bigl(r_\xx(n+1)-r_{\xx}(n)\bigr)_{n \in \N}$ is bounded.

\medskip

Let us explain that Conjecture 28 is valid  for sequences whose language is closed under reversal:  By Theorem 1 of \cite{Cassaigne1996}, the sequence  ${\mathcal C}_{\xx}(n+1)-{\mathcal C}_{\xx}(n)$ is bounded if $\xx$  has  at most linear factor complexity. If the language of a sequence is closed under reversal, then Item (b) of Theorem 17 in \cite{AlCaLiShSt2025} (cf. also \cite{BaMaPe2007}) says $${\mathcal P}_\xx(n) + {\mathcal P}_\xx(n+1) \leq 2+{\mathcal C}_{\xx}(n+1)-{\mathcal C}_{\xx}(n)$$ and hence  the sequence  $\bigl({\mathcal P}_\xx(n)\bigr)_{n \in \N}$ is bounded as well. Finally, Item (b) of Theorem 9 in~\cite{AlCaLiShSt2025} implies  for every $n \in \N$
$$r_\xx(n+1)-r_{\xx}(n)=\frac{1}{2}\left({\mathcal C}_{\xx}(n+1)+{\mathcal P}_\xx(n+1)\right)-\frac{1}{2}\left({\mathcal C}_{\xx}(n)+{\mathcal P}_\xx(n)\right).$$
Consequently,  $r_\xx(n+1)-r_{\xx}(n)$ is  bounded as it is the sum of three bounded sequences.

\medskip

The equivalence on the set of finite words introduced in \cite{AlCaLiShSt2025} uses the antimorphism $R$ which maps a finite word $w$ to its reversal. Formally, a class of equivalence is an orbit of a group $G= \{R, Id\}$ with the composition being the group operation.   Any group $G$ of operators on $\mathcal{A}^*$ can be used to introduce an equivalence in a natural way:  two   words $u$ and $v$ are $G$-equivalent, if  $ u=g(v)$ for some $g \in G$. 
We must  emphasize that the $G$-equivalence we are talking about  is only a special case of the equivalence considered by John Machacek in~\cite{Machacek2025}. 

\medskip

Let us consider the simplest case on the binary alphabet $\mathcal{A} = \{0,1\} $ and   the group $G$  generated  by $R$ and the  morphism $E$  which exchanges leters $0$ and $1$, i.e., $E(u_1u_2\cdots u_n) = (1-u_1)(1-u_2) \cdots (1-u_n)$. Note that $G = \{Id, R, E, ER\}$.    For the equivalence thus established, we can introduce a new complexity, say $G$-complexity,  as the function which counts the number of equivalence classes occurring in $\mathcal{L}_n(\uu)$.  We can  ask the same questions as for the  reflection complexity: 
\begin{enumerate}
    \item What is the $G$-complexity  for a given sequence $\xx$ over $\{0,1\}$.    In particular, consider the Thue-Morse sequence   or  a  complementary symmetric Rote sequence as the languages of these sequences are  closed under $G$.  
    \item  Can the $G$-complexity  be used to characterize eventually periodic sequences as is the case of reflection complexity?
    
\end{enumerate}

\section{Appendix - proof of Proposition \ref{pro:RSaLS}}

 \begin{proof} 
 First we show that for every $n \in \N$ there exists a left special factor $w$  of length $n$ and two different letters, say $a$ and $b$, such that  $aw$ and $bw$ occur in $\uu$ infinitely many times. 
Analogously,  for a  right special  factor.  
 
 Let $n \in \N$ be fixed. As $\mathcal{L}_{n+1}(\uu)$ is finite, there  exists an index $k_n \in \N$ such that every factor of length $n+1$ of the suffix ${\bf s} =  u_{k_n}u_{k_n+1}u_{k_n+2}\cdots$ of $\uu$ occurs in ${\bf s}$ infinitely many times. Thus the Rauzy graph $\Gamma_{n}$  of ${\bf s} $ is strongly connected. Since ${\bf s}$ is aperiodic, at least one vertex  of $\Gamma_{n}$, say  $f$, has in-degree $\geq 2$ and at least one vertex, say $g$,  has  out-degree $\geq 2$. In other words, there are two edges ending in $f$, they have the form $af$  and $bf$ for some letters $a\neq b$ and two edges starting in $g$ which have the form $gc$ and $gd$ for some $c,d \in \mathcal{A}$. The words $af$,  $bf$, $gc$ and $gd$   belong to $\mathcal{L}_{n+1}(\uu)$ and  by definition of $k_n$, they occur in $\uu$ infinitely many times.

\medskip
\noindent Let $L_n$ denote the list of all pairs $(\{a,b\}, f)$ with 
the property: 

\medskip

\centerline{ $af$ and $bf$ occur in $\uu$ infinitely many times, $a\neq b $ and 
 $|f| = n$. }

 \medskip

\noindent Let us stress that if $(\{a,b\}, f)$ occurs in $L_n$, then the pair $(\{a,b\}, f')$, where $f'$ is a prefix of $f$,  occurs in the list $L_{|f'|}$.  

As the alphabet $\mathcal{A}$ is finite, there exist two different letters, say $a_L,b_L$,  such that     $\{a,b\} = \{a_L,b_L\} $ for at least one pair $(\{a,b\}, f)$ from the list $L_n$ for every $n \in \N$.  

Finiteness of $\mathcal{A}$ guarantees that there exists a letter $x_0$ such that infinitely many pairs have the form   $(\{a_L, b_L\}, f)$ and  $x_0$ is a prefix of $f$. Analogously, there exists a letter $x_1$ such that infinitely many pairs have the  form   $(\{a_L, b_L\}, f)$ and  $x_0x_1$ is a prefix of $f$, etc. In this way we construct a~sequence ${\bf x}_L = x_0x_1x_2 \cdots$.

 \end{proof}

\end{document}